\documentclass[10pt,twoside]{siamart1116}

\usepackage[english]{babel}
\usepackage{graphicx,epstopdf,epsfig}
\usepackage{amsfonts,fancyhdr,graphics,amsmath,amssymb}
\usepackage{mathtools}
\usepackage{nicematrix}
\usepackage{bm}
\usepackage{microtype}
\usepackage[numbers]{natbib}

\newcommand{\Names}{Hemant K. Mishra}
\newcommand{\Title}{Lidskii theorem for directional derivatives of symplectic eigenvalues}

\newtheorem{remark}[theorem]{Remark}

\renewcommand{\theequation}{\thesection.\arabic{equation}}
\renewtheorem{theorem}{Theorem}[section]

\begin{document}

\bibliographystyle{plain}

% Leave these commented lines here, as requested in the ELA template
%\input{ELAheader-template.tex}
\setcounter{page}{1}
\thispagestyle{empty}

\title{\Title}

% \title{\Title\thanks{Received by the editors on \DoS.
% Accepted for publication on \DoA.
% Handling Editor: \HE. Corresponding Author: \CA}}

\author{Hemant K.\ Mishra\thanks{Department of Mathematics \& Computing,
Indian Institute of Technology (ISM) Dhanbad, India
(hemantmishra1124@iitism.ac.in).}}

\markboth{\Names}{\Title}
\maketitle

\begin{abstract}
    In this paper, we present a refinement of the symplectic Lidskii theorem for directional derivatives of symplectic eigenvalues, which is inspired by the work of Sendov \textit{[Electron. J. Linear Algebra \textbf{41} (2025), 338--341]}.
    We show that for a $2n \times 2n$ real positive definite matrix $A$ and symmetric matrices $H, K$
    \begin{align*}
        d'(A;H+K)-d'(A;H)
              \prec_{\pi_A}
              d'(A;K).
    \end{align*}
    Here $\prec_{\pi_A}$ represents block-majorization corresponding to the partition $\pi_A$ of $\{1,\ldots, n\}$ given by the equality blocks of symplectic eigenvalues of $A$.
    We know that \emph{Hermitianization} of a symmetric matrix occurs in the directional derivative expression of symplectic eigenvalues.
    We establish a componentwise inequality comparing the symplectic eigenvalues of a positive definite matrix with the ordinary eigenvalues of its associated {Hermitianization}, and also characterize the equality case.
    By using the aforementioned inequality, we then show that the symplectic Lidskii theorem follows from the directional derivative Lidskii theorem.
\end{abstract}



\section{Introduction}\label{sec:introduction}
    The classical Lidskii theorem describes perturbation of ordered spectrum of real symmetric matrices.
    More concretely, if $X,Y \in \mathbb{R}^{n\times n}$ are symmetric matrices then
    \begin{equation}\label{eq:intro-classical-lidskii}
        \lambda(X+Y)-\lambda(X)\prec \lambda(Y),
    \end{equation}
    where $\lambda(\cdot)$ denotes the eigenvalue vector in non-decreasing order and $\prec$ denotes majorization. 
    Sendov~\cite{Sendov2025} recently established a sharper form of Lidskii theorem in terms of the directional derivatives of ordered eigenvalues:  
    if $X,Y, Z \in \mathbb{R}^{n\times n}$ are symmetric matrices and $\pi(Z)$ is the partition of $\{1,\ldots,n\}$ generated by the equal eigenvalues of $Z$ in $\lambda(Z)$, then
    \begin{equation}\label{eq:intro-sendov}
        \lambda'(Z;X+Y)-\lambda'(Z;X) \prec_{\pi(Z)}\lambda'(Z;Y).
    \end{equation}
    Here $\prec_{\pi(Z)}$ denotes \emph{block-wise} majorization, which is stronger than ordinary majorization. 
    The relation in~\eqref{eq:intro-sendov} allows permutations only inside the eigenspaces associated with the repeated eigenvalues of the base matrix $Z$.  
    The classical Lidskii theorem is recovered by taking the base matrix $Z$ to be the $n \times n$ identity matrix $I_n$ in~\eqref{eq:intro-sendov}, which simply follows due to the fact that  $\lambda'(I_n;X)=\lambda(X)$.
    Sendov's argument is strikingly economical.  
    On each multiplicity block, the directional derivatives of the eigenvalues are precisely the eigenvalues of the compression of the direction matrix to the corresponding eigenspace.  
    The result \eqref{eq:intro-sendov} is obtained by applying the classical Lidskii theorem independently on every block. 

    The purpose of this paper is to carry this philosophy into symplectic spectral theory.  
    Let $J=\begin{psmallmatrix}0&I_n\\-I_n&0\end{psmallmatrix}$.
    Recall that a matrix \(S\in\mathbb{R}^{2n\times 2n}\) is symplectic if \(S^{T}JS=J\). 
    Williamson's theorem \cite{Williamson} states that for every positive definite matrix \(A\in\mathbb{R}^{2n\times 2n}\), there exists a symplectic matrix \(M\) such that
        \begin{align} \label{eq:will-diag-form}
            M^{T}AM=\begin{pmatrix}D&0\\0&D\end{pmatrix},
        \end{align}
    where $D$ is a uniquely determined diagonal matrix with diagonal entries \( 0< d_1(A) \leq \cdots \leq d_n(A)\), which are known as the symplectic eigenvalues of \(A\).
    Denote by $d(A)$ the vector consisting of the symplectic eigenvalues of $A$ in non-decreasing order.

    Jain and Mishra~\cite{mishraderivatives} proved the following symplectic analogue of Lidskii theorem holds for all positive definite matrices $A,B \in \mathbb{R}^{2n \times 2n}$
    \begin{equation}\label{eq:intro-symp-lidskii}
       d(A+B)-d(A)\prec^w d(B),
    \end{equation}
    where $\prec^w$ denotes weak supermajorization.  
    It is natural to ask whether~\eqref{eq:intro-symp-lidskii} also admits a directional-derivative-type refinement.
    In principle, it should be possible to get a refinement of the symplectic Lidskii theorem due to the availability of the necessary tools in the symplectic setting---explicit directional derivative of ordered symplectic eigenvaluess and Lidskii theorem.
    However, the symplectic setting presents several structural obstacles which make a direct adaptation of Sendov's proof impossible.  
    
    There are three main obstacles worth isolating because they also explain the statement and the proof of our main results.
    Let $A, H \in \mathbb{R}^{2n\times 2n}$ such that $A$ is positive definite and $H$ is symmetric.
    The directional derivative of the \(j\)-th symplectic eigenvalue at $A$ in the direction \(H\) is defined by
    \begin{align}\label{eq:dir-deri-symplectic}
        d_j'(A;H) = \lim_{t \searrow 0}
        \frac{d_j(A+tH)-d_j(A)}{t}.
    \end{align}
    The limit in~\eqref{eq:dir-deri-symplectic} exists and its explicit description is given as follows.
    Suppose that $d_p(A)=\cdots=d_q(A)$ is a symplectic eigenvalue of multiplicity \(r=q-p+1\). 
    A symplectic eigenbasis associated with this eigenspace determines a \(2r\times 2r\) \emph{symplectic-compression} of \(H\), which we write in block form as
    \begin{align}
        \begin{pmatrix}
            X&Y\\
            Y^{T}&Z
        \end{pmatrix}.
    \end{align}
    \sloppy
    The \(r\) directional derivatives corresponding to this repeated symplectic eigenvalue are obtained from the ordinary eigenvalues of the corresponding \(r\times r\) Hermitian matrix
    \begin{align}
    \mathcal{H}_{p,q}(H)\coloneqq\frac12\left[
        X+Z
        +\mathrm{i}\bigl(Y-Y^{\top}\bigr)
        \right]
    \end{align}
    arranged in the non-decreasing order by
    \begin{align}
        \bigl(d_p'(A;H),\ldots,d_q'(A;H)\bigr)= \lambda^\uparrow\left( \mathcal{H}_{p,q}(H)
                \right).
    \end{align}

    \vspace{0.5cm}
    
    \noindent \textbf{$1.$ Directional derivatives are ordinary eigenvalues rather than symplectic eigenvalues:}
    We know that for ordinary eigenvalues, the directional derivative $\lambda'(Z;X)$ on a multiplicity block consists of the eigenvalues of the compression of the direction matrix $X$ on the corresponding eigenspace of $Z$.
    The directional derivative for symplectic eigenvalues $d'(A;H)$ on a multiplicity block is given by the eigenvalues of a complex Hermitian matrix $\mathcal{H}(H)$ corresponding to $H$.
    Thus, the local symplectic problem is converted into an ordinary Hermitian spectral problem on every block.  
    This presents a challenge to meaningfully bring symplectic eigenvalues into the picture to justify refinement of the symplectic Lidskii theorem.
    \vspace{0.5cm}
    
    \noindent \textbf{$2.$ Local and global majorization relations are different:}
    Sendov's local refinement uses ordinary majorization, which is compatible with the global majorization relation in Lidskii theorem.  
    The symplectic Lidskii theorem~\eqref{eq:intro-symp-lidskii} uses weak supermajorization. 
    A satisfactory directional derivative result, which would naturally contain block majorization relations, should therefore explain how it descends to the weak-supermajorization relation satisfied by the symplectic eigenvalues of positive definite directional matrices and their sum.

    \vspace{0.5cm}
    
    \noindent \textbf{$3.$ There is no trivial base point:}
    Sendov recovers classical Lidskii by setting the base matrix $Z$ in~\eqref{eq:intro-sendov} equal to the identity matrix $I_n$.  
    For symplectic eigenvalues, the directional derivative $d'(I_{2n};B)$ is given by the eigenvalue vector $\lambda^\uparrow(\mathcal H(B))$ of the $n \times n$ Hermitian matrix corresponding to $B$.
    For a general positive definite matrix $B$, this vector is generally different from $d(B)$, and hence a mere formal specialization of the directional derivative relation would not yield the symplectic Lidskii theorem.   
    A separate comparison between $d(B)$ and the eigenvalues of $\mathcal H(B)$ is required, which is one of the key results of our work.

    Our first result establishes the exact symplectic counterpart of Sendov’s block-majorization theorem. We then prove a componentwise comparison between the symplectic spectrum of a positive definite matrix and the spectrum of its associated Hermitianization, together with a characterization of the equality case. These results provide the missing bridge between ordinary majorization arising from the local Hermitian problem and weak supermajorization of the global symplectic spectrum. 
    As a consequence, we show that the symplectic Lidskii theorem follows directly from our directional-derivative theorem.
    Interestingly, the relationship with Sendov's result goes beyond analogy. 
    By embedding a real symmetric matrix \(Z\) in a positive definite matrix of the form \((Z+tI_n)\oplus(Z+tI_n)\), our directional theorem reduces exactly to Sendov's theorem. 
    Thus, the result obtained here simultaneously contains the ordinary directional Lidskii theorem and refines the symplectic Lidskii inequality.
\section{Main results and proofs}
    Let us set some notations and recall some basic definitions and results that will be helpful in the further development.
    For $x\in\mathbb{R}^n$, let $x^\uparrow$ and $x^\downarrow$ denote the vectors obtained by arranging the entries of $x$ in non-decreasing and non-increasing order, respectively.  
    For $x,y\in\mathbb{R}^n$, we say $x$ is majorized by $y$ and write $x \prec y$ if
    \begin{align}
       \sum_{j=1}^k x_j^\downarrow
       \leq
       \sum_{j=1}^k y_j^\downarrow,
       \qquad k=1,\ldots,n-1,
    \end{align}
    and equality holds when $k=n$.  
    We say $x$ is weakly supermajorized by $y$ and write $x \prec^w y$ if
    \begin{equation}\label{eq:weak-supermajorization}
       \sum_{j=1}^k x_j^\uparrow
       \geq
       \sum_{j=1}^k y_j^\uparrow,
       \qquad k=1,\ldots,n.
    \end{equation}
    Let $\pi=\{I_1,\ldots,I_s\}$ be a partition of $\{1,\ldots,n\}$.  
    We say $x$ is block-majorized by $y$ and write $x \prec_\pi y$ if
    \begin{align}
       x_{I_\alpha}\prec y_{I_\alpha},
       \qquad \alpha=1,\ldots,s.
    \end{align}
    Here $x_{I_\alpha}$ denotes the sub-vector of $x$ consisting of elements of $x$ corresponding to the indices in $I_\alpha$.
    The relation $x \prec_\pi y$ is equivalent to $x$ belonging to the convex hull of vectors obtained from $y$ by permutations that preserve every block of $\pi$.  
    It will serve us well to note the following easy to verify implications
    \begin{equation}\label{eq:pi-implies}
       x\prec_\pi y
       \quad\Longrightarrow\quad
       x\prec y
       \quad\Longrightarrow\quad
       x\prec^w y.
    \end{equation}
    Throughout the paper, we fix a positive definite matrix $A\in\mathbb{R}^{2n \times 2n}$ and choose a symplectic matrix $M$ such that~\eqref{eq:will-diag-form} holds.
    Suppose $A$ has $s$ distinct symplectic eigenvalues $\mu_1 < \cdots < \mu_s$.
    Define a partition $\pi_A \coloneqq \{I_1,\ldots, I_s\}$ corresponding to the multiplicity blocks of the symplectic eigenvalues, i.e.,
    \begin{align}
        I_\alpha = \{j: d_j(A) = \mu_\alpha\}, \qquad \alpha =1,\ldots, s.
    \end{align}
    Let the columns of $M$ be given by $u_1,\ldots, u_n, v_1, \ldots, v_n$ in order.
    For each multiplicity block $I_\alpha$ define
    \begin{align}
       M_\alpha
       \coloneqq[u_j:j\in I_\alpha,\ v_j:j\in I_\alpha].
    \end{align}
    For every symmetric matrix $H\in \mathbb{R}^{2n \times 2n}$, we call $M_\alpha^{\top}HM_\alpha$ the \emph{symplectic-compression} of $H$ by $M_\alpha$.
    Suppose the block structure of the symplectic compression is given by
    \begin{equation}\label{eq:block-compression}
     M_\alpha^{\top}HM_\alpha
        =
         \begin{bmatrix}
            H_{11}^{(\alpha)}&H_{12}^{(\alpha)}\\
            (H_{12}^{(\alpha)})^{\top}&H_{22}^{(\alpha)}
         \end{bmatrix},
    \end{equation}
    where each block has size $|I_\alpha| \times |I_\alpha|$.
    Consider the \emph{Hermitianization} of the symplectic compression $M_\alpha^{\top}HM_\alpha$
        \begin{equation}\label{eq:local-hermitianization}
             \mathcal H_\alpha(H)
                 \coloneqq \frac12\left[
                 H_{11}^{(\alpha)}+H_{22}^{(\alpha)}
                 +\mathrm{i}\left(
                 H_{12}^{(\alpha)}-(H_{12}^{(\alpha)})^{\top}
                 \right)
                 \right].
        \end{equation}
    Here $\mathrm{i} \coloneqq \sqrt{-1}$.
    The map $H\mapsto\mathcal H_\alpha(H)$ is linear.
    In a prior work \cite[Theorem~3.6]{mishrafirst}, we established that\footnote{There is a minus sign that appears in the directional derivative expression $d_j^{'}(A;H)$ in \cite{mishrafirst}. Our presentation absorbs the minus sign inside the eigenvalue argument.}
    \begin{equation}\label{eq:mishra-block-formula}
       d'(A;H)_{I_\alpha}
       =\lambda^\uparrow\!\left(\mathcal H_\alpha(H)\right).
    \end{equation}

    Our first main result is the following direct symplectic analogue of Sendov's theorem~\eqref{eq:intro-sendov}.
    \begin{theorem}[Symplectic Lidskii for directional derivatives]\label{thm:directional-lidskii}
        Let $A \in \mathbb{R}^{2n \times 2n}$ be a positive definite matrix.
        For all symmetric matrices $H,K \in \mathbb{R}^{2n \times 2n}$, we have
        \begin{equation}\label{eq:main-directional}
              d'(A;H+K)-d'(A;H)
              \prec_{\pi_A}
              d'(A;K).
        \end{equation}
        % Consequently,
        % \begin{equation}\label{eq:main-directional-ordinary}
        %   d'(A;X+Y)-d'(A;Y)
        %   \prec
        %   d'(A;X),
        % \end{equation}
        % and therefore
        % \begin{equation}\label{eq:main-directional-weak}
        %   d'(A;X+Y)-d'(A;Y)
        %   \prec^w
        %   d'(A;X).
        % \end{equation}
    \end{theorem}
    \begin{proof}
        By using the relation \eqref{eq:mishra-block-formula}, linearity of the map $H \mapsto \mathcal{H}_\alpha(H)$, and the classical Lidskii theorem, we get
        \begin{align}
            d'(A;H+K)_{I_\alpha}- d'(A;H)_{I_\alpha}
                &=\lambda^\uparrow\!\left(\mathcal H_\alpha(H+K)\right)-\lambda^\uparrow\!\left(\mathcal H_\alpha(H)\right) \\
                &= \lambda^\uparrow\!\left(\mathcal H_\alpha(H)+\mathcal H_\alpha(K)\right)-\lambda^\uparrow\!\left(\mathcal H_\alpha(H)\right) \\
                &\prec \lambda^\uparrow\!\left(\mathcal H_\alpha(K)\right) \\
                &= d'(A;K)_{I_\alpha}.
        \end{align}
        This completes the proof.
    \end{proof}
    \begin{remark}
        It is interesting to note that
        Sendov's directional derivative result~\eqref{eq:intro-sendov} is a special case of Theorem~\ref{thm:directional-lidskii}.
        Indeed, let $X, Y, Z \in \mathbb{R}^{n \times n}$ be symmetric matrices.
        Choose $t > 0$ large enough so that $A \coloneqq (Z \oplus Z) + t I_{2n}$ is positive definite.
        Let $H \coloneqq X \oplus X$ and $K \coloneqq Y \oplus Y$.
        We have $d_j(A) = \lambda_j^\uparrow(Z)+t$ for all $j$, so that $\pi_A = \pi(Z)$.
        Also, $d'(A; H+K)= \lambda'(Z; X+Y)$, $d'(A; H)= \lambda'(Z; X)$, and $d'(A; K)= \lambda'(Z; Y)$.
        The symplectic directional derivative relation~\eqref{eq:main-directional} then reduces to \eqref{eq:intro-sendov}.
    \end{remark}
    An immediate consequence of the previous result is the following differentiability property for simple symplectic eigenvalues established in \cite{mishraderivatives}.
    \begin{corollary}\label{cor:simple-spectrum}
        If the symplectic eigenvalues of $A$ are all distinct, then for every symmetric
            $H, K\in\mathbb{R}^{2n \times 2n}$,
            \begin{equation}\label{eq:additivity-simple}
               d'(A;H+K)=d'(A;H)+d'(A;K).
            \end{equation}
    \end{corollary}
    \begin{proof}
        If all symplectic eigenvalues of $A$ are distinct, then every block of $\pi_A$ is a singleton set.  Majorization on a one-dimensional block is equality.  Theorem~\ref{thm:directional-lidskii} thus yields~\eqref{eq:additivity-simple}.        
    \end{proof}
    
     We now present an analysis to show that the block-majorization relation \eqref{eq:main-directional} is finer than the symplectic Lidskii theorem~\eqref{eq:intro-symp-lidskii}.
     We begin by declaring some notations and proving some preliminary results.
     Let
     \begin{align}
         W_+\coloneqq\frac1{\sqrt2}\begin{bmatrix}I_n\\-\mathrm{i}I_n\end{bmatrix},
         \qquad
         W_-\coloneqq\frac1{\sqrt2}\begin{bmatrix}I_n\\\mathrm{i}I_n\end{bmatrix},
         \qquad
         Q\coloneqq[W_+\ W_-].
        \end{align}
        Then $Q$ is unitary and it is easy to verify that 
        \begin{equation}\label{eq:diag-iJ}
             S \coloneqq Q^*(\mathrm{i}J)Q
           =\begin{bmatrix}I_n&0\\0&-I_n\end{bmatrix}
           .
        \end{equation}
        % Also, for every symmetric matrix $H \in \mathbb{R}^{2n \times 2n}$, recall that ${\mathcal H(H)}=\frac12\left[H_{11}+H_{22}-\mathrm{i}(H_{12}-H_{12}^{\top})\right]$. It is easy to verify that the leading $n \times n$ block of $\widehat H \coloneqq Q^* H Q$ is the complex conjugate of ${\mathcal H(H)}$; i.e.,
        % we have
        % \begin{align}\label{eq:hat-h-blocks}
        %     \widehat H = 
        %     \begin{bNiceMatrix}[margin]
        %       \frac12\left[H_{11}+H_{22}+\mathrm{i}(H_{12}-H_{12}^{\top})\right] & \Block{1-1}{\Large *} \\
        %       \Block{1-1}{\Large *} & \Block{1-1}{\Large *}
        %     \end{bNiceMatrix}.
        % \end{align}
        We shall use the notation $\# \mathcal{S}$ to denote the cardinality of a set $\mathcal{S}$.
    \begin{lemma}\label{lem:inertia-count}
            Let $H\in\mathbb{R}^{2n \times 2n}$ be a positive definite matrix.
            For every $t>0$ that is not a symplectic eigenvalue of $H$, we have
            \begin{equation}\label{eq:inertia-count}
               n_-\!\left(\widehat H-tS\right)
               =\#\{j:d_j(H)<t\},
            \end{equation}
            where $S$ is defined in \eqref{eq:diag-iJ} and $n_-(T)$ denotes the number of negative eigenvalues of a Hermitian matrix $T$, counted with multiplicity.
        \end{lemma}
    \begin{proof}
    Since $\widehat H-tS
         =\widehat H^{1/2}
         \left(I_{2n}-t\widehat H^{-1/2}S\widehat H^{-1/2}\right)
         \widehat H^{1/2}$,
    the Sylvester's law of inertia gives
    \begin{align}\label{eq:negative-eigenvalues-count}
        n_-\!\left(\widehat H-tS\right) = n_-\!\left(I_{2n}-t\widehat H^{-1/2}S\widehat H^{-1/2}\right).
    \end{align}
    We have $\widehat H^{-1/2}S\widehat H^{-1/2}=Q^* H^{-1/2} \mathrm{i}J H^{-1/2} Q$.
    Also, we know that the eigenvalues of $H^{-1/2} \mathrm{i}J H^{-1/2}$ are given by the inverses of $\pm d_j(H)$ for $j=1,\ldots, n$ (see, e.g., \cite{mishraderivatives}).
    Therefore, the eigenvalues of $I_{2n}-t\widehat H^{-1/2}S\widehat H^{-1/2}$ are given by
    \begin{align}
        1 \pm \dfrac{t}{d_j(H)}, \qquad j=1,\ldots, n.
    \end{align}
    This implies that the negative eigenvalues of $I_{2n}-t\widehat H^{-1/2}S\widehat H^{-1/2}$ correspond to the symplectic eigenvalues $d_j(H)$ that satisfy $d_j(H) < t$.
    We conclude the proof by using the relation \eqref{eq:negative-eigenvalues-count}.
    \end{proof}

    We declare that inequalities between vectors are understood componentwise throughout the paper.
    For a symmetric matrix $H \in \mathbb{R}^{2n \times 2n}$ in the block form
    \begin{align}\label{eq:h-block-form}
    H= \begin{psmallmatrix}
        H_{11} & H_{12} \\
        H_{12}^\top & H_{22}
    \end{psmallmatrix}
    \end{align}
    with each block of size $n \times n$, define a corresponding Hermitian matrix by
    \begin{align} \label{eq:corr-hermitian}
        \mathcal{H}(H) \coloneqq \dfrac{1}{2} \left( H_{11}+H_{22} + \mathrm{i}(H_{12}-H_{12}^\top) \right).
    \end{align}
    The Hermitianization of a symmetric matrix appears in the directional derivative formula~\eqref{eq:mishra-block-formula} of symplectic eigenvalues.
    The next result gives a relationship between the symplectic eigenvalues and ordinary eigenvalues of the Hermitianization of a positive definite matrix, which plays a key role in deriving the symplectic Lidskii theorem from the directional derivative Lidskii theorem. 
     \begin{theorem}\label{thm:hermitianization}
        For every $2n \times 2n$ real positive definite matrix $H$, we have
            \begin{equation}\label{eq:componentwise-main}
               d(H)
               \leq
               \lambda^\uparrow(\mathcal H(H)).
            \end{equation}
        Furthermore, equality holds in \eqref{eq:componentwise-main} if and only if $HJ = JH$.
    \end{theorem}
    \begin{proof}
            % Since $\mathcal{H}(H)$ is a Hermitian matrix, proving the inequality \eqref{eq:componentwise-main} implicitly also proves that the matrix is positive definite matrix.
            % First observe that
            % \begin{align}
            %   \mathcal H(H)=W_-^*HW_-.
            % \end{align}
            % Since $H>0$ and $W_-$ has full column rank, $\mathcal H(H)>0$.
        Let $H$ be a $2n \times 2n$ real positive definite matrix and let
          $c_1\leq\cdots\leq c_n$ be the eigenvalues of its Hermitianization $\mathcal H(H)$.  
        % By Lemma~\ref{lem:complex-basis}, these are also the eigenvalues of $\mathcal H(H)$.
        Fix $t>0$ that is neither a symplectic eigenvalue of $H$ nor an eigenvalue of $\mathcal{H}(H)$.  
        It is easy to verify that $\widehat H$ has the following block structure: 
        \begin{align}\label{eq:hath-block-form}
            \widehat H = 
                \begin{bNiceMatrix}[margin]
                  \overline{\mathcal{H}(H)} & E \\
                  E^* & {\mathcal{H}(H)}
                \end{bNiceMatrix},
        \end{align}
        where $E = \frac{1}{2} (H_{11}-H_{22}+\mathrm{i}(H_{12}+H_{12}^\top))$.
        ${\mathcal{H}(H)}$ being a principal submatrix of $\widehat H$ is a Hermitian positive definite matrix.
        On the other hand,
        \begin{align}
             \widehat H-tS
             =\begin{bmatrix}
             \overline{\mathcal{H}(H)}-tI_n&E\\
             E^*& {\mathcal{H}(H)} +tI_n
         \end{bmatrix}.
        \end{align}
        We know that ${\mathcal{H}(H)}$ and $\overline{\mathcal{H}(H)}$ being Hermitian complex-conjugates of each other have the same eigenvalues.
        Since $\overline{\mathcal{H}(H)}-tI_n$ is a principal submatrix of $\widehat H - tS$, Cauchy's interlacing theorem implies that
        \begin{align}\label{eq:schur-inertia}
             n_-(\widehat H-tS)
             &\geq n_-(\overline{\mathcal{H}(H)}-tI_n) \\
             &= n_-({\mathcal{H}(H)}-tI_n).\label{eq:neg-eig-count-hat-h-st}
        \end{align}
        The relation \eqref{eq:neg-eig-count-hat-h-st} thus gives
        \begin{align}\label{eq:neg-eig-ineq}
            n_-(\widehat H-tS)
             &\geq \#\{j:c_j<t\}.
        \end{align}
        Lemma~\ref{lem:inertia-count} states that $n_-(\widehat H-tS)
             =\#\{j:d_j(H)<t\}$.
        Together with~\eqref{eq:neg-eig-ineq}, we have
        \begin{equation}\label{eq:count-comparison}
         \#\{j:d_j(H)<t\}
         \geq
         \#\{j:c_j<t\}
        \end{equation}
        for all such $t$.
    
        Suppose for contradiction that $d_j(H)>c_j$ for some $j$.  
        Choose $t$ with $c_j<t<d_j(H)$
        and away from the two finite spectra.  
        Then $\#\{\ell:c_\ell<t\}\geq j$ whereas $\#\{\ell:d_\ell(H)<t\}\leq j-1$
        contradicting~\eqref{eq:count-comparison}.  
        Therefore
        \begin{align}
           d_j(H)\leq c_j
           =\lambda_j^\uparrow(\mathcal H(H)), \qquad j=1,\ldots, n.
        \end{align}
    
        We now prove the necessary and sufficient condition for the equality to hold in \eqref{eq:componentwise-main}.
        If $H$ commutes with $J$ then the blocks of $H$ satisfy $H_{11}=H_{22}$ and $H_{12}^\top = -H_{12}$.
        If a unitary matrix $U+\mathrm{i}V$, with $U, V \in \mathbb{R}^{n \times n}$ such that $(U+\mathrm{i}V)^* \mathcal{H}(H)(U+\mathrm{i}V)= D$ is a diagonal matrix then we have 
        \begin{align}
            \begin{pmatrix}
            U & V \\
            -V & U
        \end{pmatrix}^\top H 
        \begin{pmatrix}
            U & V \\
            -V & U
        \end{pmatrix}
        = \begin{pmatrix}
            D & 0 \\
            0 & D
        \end{pmatrix}.
        \end{align}
        This gives $d(H)=\lambda^\uparrow(\mathcal{H}(H))$.
    
        Conversely, suppose that the symplectic eigenvalues of $H$ are the same as the eigenvalues of $\mathcal{H}(H)$.
        % We already know that
        % \begin{align}
        % d_j(H)\leq \lambda_j^\uparrow(\mathcal{H}(H)),
        % \qquad j=1,\ldots,n.
        % \end{align}
        By Williamson's theorem and the fact that symplectic matrices have determinant equal to one \cite{dms}, we have
        \begin{align}\label{eq:det-h-symp-eig}
            \det \widehat H = \det H =
            \prod_{j=1}^n d_j(H)^2.
        \end{align}
        Taking the Schur complement with respect to the lower-right block of \eqref{eq:hath-block-form} yields
        \begin{align}
             \det \widehat H
                &=
                \det\mathcal{H}(H)\,
                \det\left(
                \overline{\mathcal{H}(H)}
                -
                E\mathcal{H}(H)^{-1}E^*
                \right) \\
                &\leq \det\mathcal{H}(H)
                \det
                \overline{\mathcal{H}(H)} \label{eq:schur-det-ineq}\\
                &= \det\mathcal{H}(H)^2 \\
                &= \prod_{j=1}^n \lambda_j^{\uparrow}(\mathcal{H}(H))^2.
                \label{eq:schur-det}
        \end{align}
        The assumption $d(H)=\lambda^\uparrow(\mathcal{H}(H))$ thus forces the inequality \eqref{eq:schur-det-ineq} to be equality, which is possible only if $E=0$.
        We then have $H_{11}=H_{22}$ and $H_{12}=-H_{12}^\top$, which implies that $H$ commutes with $J$.
    \end{proof}

    \begin{corollary}\label{cor:inequality-dir-der}
        For every positive definite matrix $H\in\mathbb{R}^{2n\times 2n}$,
        \begin{equation}\label{eq:identity-componentwise}
           d(H) \leq d'(I_{2n};H).
        \end{equation}
    \end{corollary}
    \begin{proof}
        It follows directly from Theorem~\ref{thm:hermitianization} and the directional derivative expression $d'(I_{2n}; H)= \lambda^\uparrow (\mathcal{H}(H))$.
    \end{proof}

    % The componentwise comparison can be transported from the identity to an arbitrary base point at the cost of passing from coordinatewise order to weak supermajorization.

    We are now set to show that the symplectic Lidskii theorem comes as a special case of Theorem~\ref{thm:directional-lidskii}.
    \begin{corollary}[Symplectic Lidskii]\label{cor:symp-lidskii}
            Let $H, K$ be $2n \times 2n$ real positive definite matrices.
            We have
            \begin{align} \label{eq:symp-lidsii}
                d(H+K)-d(H) \prec^w d(K).
            \end{align}
        \end{corollary}
    \begin{proof}
        Choose a symplectic matrix \(M\) which gives a Williamson diagonalization of \(H+K\), with the symplectic eigenvalues arranged in non-decreasing order. 
        Choose $A \coloneqq M^{-T}M^{-1}$, which has equal symplectic eigenvalues so that $\prec_{\pi_A} = \prec$.
        By Theorem~\ref{thm:directional-lidskii}, we thus have 
        \begin{align}\label{eq:dir-der-lidskii}
            d'(A;H+K)-d'(A;H) \prec d'(A;K).
        \end{align}
        We first observe that
        \begin{align} \label{eq:dprime-h+k}
            d'(A;H+K)=d(H+K).
        \end{align}
        which follows directly from the first principle~\eqref{eq:dir-deri-symplectic} and the fact that both $A$ and $H+K$ are simultaneously diagonalizable in the sense of Williamson's theorem.
        Using the fact $M^\top A M = I_{2n}$ and Corollary~\ref{cor:inequality-dir-der}, we thus get
        \begin{align} \label{eq:dprime-ah-inequality}
            d'(A; H) 
                &= d'(I_{2n}; M^\top H M) \geq d(M^\top H M) = d(H).
        \end{align}
        Substituting from \eqref{eq:dprime-h+k} and \eqref{eq:dprime-ah-inequality} into \eqref{eq:dir-der-lidskii}, we thus get
        \begin{align}\label{eq:second-last-lidskii-step}
            d(H+K)- d(H) \prec^w d'(A;K).
        \end{align}
        Here we used the fact that if $x \prec y$ and $z \geq x$ then $z \prec^w y$.
        The inequality in \eqref{eq:dprime-ah-inequality} also gives $d'(A;K) \geq d(K)$.
        Applying this inequality to \eqref{eq:second-last-lidskii-step} thus gives the desired relation
        \begin{align}
            d(H+K)- d(H) \prec^w d(K).
        \end{align}
    \end{proof}

\section*{Conclusion}
    We established a Lidskii type theorem for directional derivatives of symplectic eigenvalues that involves block majorization with respect to the partition induced by the symplectic eigenvalues of the base matrix.
    As a result that is also of independent interest, we proved a componentwise inequality between the symplectic eigenvalues of a positive definite matrix and the ordinary eigenvalues of its Hermitianization; and we also showed that the precise condition for the equality case is that the matrix commutes with $J$.
    By using the aforementioned inequality, we showed that the symplectic Lidskii theorem follows from the directional derivative Lidskii theorem.

\section*{Acknowledgment}
    HKM acknowledges support from FRS Project No.~MISC~0147.
    OpenAI’s ChatGPT (GPT-5.6 Sol) was used to formulate the proof of Theorem~\ref{thm:hermitianization} which was instrumental in deriving the Lidskii theorem from the directional derivative Lidskii theorem; and it was also used to improve the presentation of the paper.

\bibliography{reference}

\end{document}